\documentclass[11pt, reqno]{amsart} 
\usepackage{amsxtra}
\usepackage{smartdiagram}
\usepackage{tikz}
\usetikzlibrary{calc,matrix,arrows,shapes}
\usepackage{pgfplots}
\tikzstyle{bubble}=[align=center, draw,inner sep=0pt,shape=ellipse,minimum height=1.25cm, minimum width=3cm]
\theoremstyle{plain}

\usepackage{graphicx} 

\usepackage{booktabs} 
\usepackage{array} 
\usepackage{paralist} 
\usepackage{verbatim} 
\usepackage{subfig} 
\usepackage{hyperref}
\usepackage{tikz}
\usepackage{lipsum}
\usepackage{dsfont}

\def\sideremark#1{\ifvmode\leavevmode\fi\vadjust{\vbox to0pt{\vss
\hbox to 0pt{\hskip\hsize\hskip1em%
\vbox{\hsize2cm\tiny\raggedright\pretolerance10000%
\noindent {\color{red}{#1}}\hfill}\hss}\vbox to8pt{\vfil}\vss}}}%

\usepackage{amsmath}
\usepackage{amssymb}
\usepackage{amsthm}
\usepackage{mathrsfs}
\usepackage{enumerate}
\usepackage[all]{xy}
\usepackage{todonotes}
\usepackage{xkvltxp}
\usepackage{fixme}
\usetikzlibrary{calc}

\newtheorem{thm}{Theorem}[section]
\newtheorem{prop}[thm]{Proposition}
\newtheorem{cor}[thm]{Corollary}

\newtheorem{lem}[thm]{Lemma}
\theoremstyle{definition}

\newtheorem{remk}[thm]{Remark}

\newcommand{\tn}[1]{\textnormal{#1}}

\newcommand{\wt}[1]{\widetilde{#1}}

\newcommand{\vocab}[1]{\textbf{#1}}

\newcommand{\vspan}{\tn{span}}

\newcommand{\diff}{\backslash}

\newcommand{\tr}{\textnormal{tr}}

\newcommand{\F}{\mathbb{F}}

\begin{document}

\title[Disconnected Cliques in Derangement Graphs]
{Disconnected Cliques in Derangement Graphs}
\author[Sara Anderson]{Sara Anderson}
\email{sara.anderson@cgu.edu}
\address{ 
Institite of Mathematical Sciences \\
Claremont Graduate University \\
Claremont, CA 91711\\
U.S.A}

\author[W. Riley Casper]{W. Riley Casper}
\email{wcasper@fullerton.edu}
\address{
Department of Mathematics \\
California State University \\
Fullerton, CA 92831 \\
U.S.A.}

\author[Sam Fleyshman]{Sam Fleyshman}
\email{fleyshmansl@csu.fullerton.edu}
\address{
Department of Mathematics \\
California State University \\
Fullerton, CA 92831 \\
U.S.A.}

\author[Matt Rathbun]{Matt Rathbun}
\email{mrathbun@fullerton.edu}
\address{
Department of Mathematics \\
California State University \\
Fullerton, CA 92831 \\
U.S.A.}
\date{}
\begin{abstract}
We obtain a correspondence between pairs of $N\times N$ orthogonal Latin squares and pairs of disconnected maximal cliques in the derangement graph with $N$ symbols.
Motivated by methods in spectral clustering, we also obtain modular conditions on fixed point counts of certain permutation sums for the existence of collections of mutually disconnected maximal cliques.
We use these modular obstructions to analyze the structure of maximal cliques in $X_N$ for small values of $N$.  We culminate in a short, elementary proof of the nonexistence of a solution to Euler's $36$ Officer Problem.
\end{abstract}

\maketitle
\section{Introduction}
In 1782, Euler published the first mathematical analysis of so-called Graeco-Latin squares, now known as orthogonal Latin squares, \emph{Recherches sur une Nouvelle Esp\`{e}ce de Quarr\'{e}s Magiques} \cite{euler}. He proved that orthogonal Latin squares of size $2n+1$ and $4n$ always exist, and hypothesized (incorrectly) that orthogonal Latin squares of size $4n+2$ were impossible. Intriguingly, it is only in sizes 2 and 6 that orthogonal Latin squares cannot exist. A rigorous proof that orthogonal Latin squares of size $N=6$ do not exist was claimed by Clausen \cite{Clausen}, but the manuscript is lost to history; the first verifiable proof of the insolvability of the problem is due to Tarry \cite{tarry}. Tarry's proof involved a case analysis of 17 families of Latin squares, and 9408 separate cases. Since then, several mathematicians and computer scientists have considered the problem. Pertaining to Euler's original conjecture, in 1984, Stinson \cite{stinson} provided a 3-page combinatorial (dis)proof. More recently, Chen and Wang in 2018 \cite{ChenWang} used quasi-difference matrices, and Ward in 2019 \cite{Ward} approached it using $(n,k)$-nets. 

While these proofs are short and elegant, they all rely on considerable expertise in the field of combinatorial design theory. We propose a novel perspective on mutually orthogonal Latin squares (MOLS) by demonstrating an original correspondence between pairs of $N \times N$ orthogonal Latin squares and pairs of disconnected maximal cliques in the derangement graph of $S_N$. We draw connections between MOLS and network analysis, casting the problems in a new light and thereby admitting additional, powerful tools into the area. In particular, spectral analysis provides naturally motivated insight into the clique structure of the derangement graph. As the kernel of the Laplacian matrix is associated with the connected components of a graph, highly interconnected regions separate from other interconnected regions in the derangement graph should align well with the eigenspace of the lowest eigenvalues of its Laplacian. As the representation theory of $S_N$ is well-understood, the eigendata of the derangement graph can be nicely leveraged in a way that, we hope, will be accessible to a wider range of mathematicians from different fields, bringing a wealth of knowledge about network analysis to bear on combinatorial questions about MOLS. 
 
Beyond Euler's original problem, constructions and investigations of sets of MOLS are ongoing areas of active research. Applications and interest include experimental design (\cite{VAG}, \cite{Mont}), coding and quantum information theory (\cite{dougherty}, \cite{HallRao}), and cryptography (\cite{HuaZhuChen}, \cite{Vaud}). Investigations into both the existence and construction of sets of higher order orthogonal Latin squares continues to be quite active (e.g., \cite{DrakeMyrvold}, \cite{ElMesShaaban}, \cite{RBCS}). As late as 2023, Egan and Wanless enumerated how many Latin squares (up to order 9) have orthogonal pairs (and how many they have) (\cite{EganWanless}). While an example of a pair of orthogonal Latin squares of order 10 graced the cover of \emph{Scientific American} in 1959 (\cite{SciAmer}), it remains unknown what the largest number of MOLS is --  there are at most 6, and Gill and Wanless recently made progress towards showing there are at most 3 (\cite{GillWanless}). Still, despite considerable efforts (\cite{McKayMeynertMyrvold}), no set of three MOLS of order 10 have yet been found.

It is expected that the lens of this paper and the techniques used can be harnessed for further analysis of questions related to higher order MOLS. We explore some of the immediate consequences for small values of $N$, and indicate, throughout the paper, potential for more elaborate applications. The heart of the paper is Section \ref{section:MaxCliques}, where we lay out the correspondence between disconnected maximal cliques in the derangement graph and orthogonal Latin squares, describe the spectral clustering motivation, and employ these ideas to provide modular obstructions to the existence of disconnected maximal cliques. In Section \ref{section:SmallN}, we showcase some of the conclusions that can be drawn in the toy cases of $N=3, 4,$ and $5$, to show the potential utility of the methods. Finally, in Section \ref{section:Euler36}, we turn to Euler's 36 Officer Problem, and provide a short, elementary proof with tools that will be more native to non-combinatorialists.


\section{Maximal Cliques in the Derangement Graph} \label{section:MaxCliques}
The \vocab{derangement graph} $X_N$ with symbols $\{1,\dots, N\}$ is the undirected graph whose vertices are the elements of $S_N$, where there is an edge between $\sigma,\tau\in S_N$ if and only if $\sigma\tau^{-1}$ is a derangement.
In other words, $X_N$ is the Cayley graph of the symmetric group $S_N$ with edges defined by derangements.

A \vocab{clique} in a graph $X$ is a collection of vertices whose induced subgraph is a complete subgraph of $X$.  Often, this induced subgraph is itself also referred to as a clique.  A clique is called a \vocab{maximal clique} if it is not a proper subgraph of any other clique in $X$.
We say that two cliques $C$ and $\wt C$ in $X$ are \vocab{disconnected} if there are no edges in $X$ between vertices in $C$ and vertices in $\wt C$.
Often the members of a clique will be enumerated, in which case we call the clique an \vocab{ordered clique}.

The derangement graph has the property that all maximal cliques have the same size.  This is equivalent to the fact that any partial Latin square can be completed, which was first proved by Hall in \cite{HallORS} and is a consequence of Hall's Marriage Theorem.
For convenience, we include an elementary proof here in our context of maximal cliques.
\begin{prop}
The maximal cliques in $X_N$ all have $N$ elements.
\end{prop}
\begin{proof}
Suppose that $C = \{\sigma_i\}_{i=1}^m\subseteq S_N$ is a maximal clique in $X_N$.
Then $\sigma_i(1)\neq \sigma_j(1)$ for all $1\leq i < j\leq m$, and therefore $m\leq N$.

Now suppose that $m < N$.
For each $1\leq k\leq N$, define
$$B_k = \{1,\dots, N\}\diff \{\sigma_1(k),\dots,\sigma_m(k)\}.$$
Then, for each $k$, $B_k$ contains $N-m$ elements.
Furthermore, each integer between $1$ and $N$ occurs in exactly $N-m$ of the sets $B_1,\dots, B_N$.
Therefore, for every subset $I\subseteq \{1,\dots, N\}$, the fibers of the function
$$\{(x,k): k\in I,\ x\in B_k\}\rightarrow\bigcup_{j\in I} B_j,\ \ (x,k)\mapsto x$$
have cardinality at most $N-m$.
Hence $\bigcup_{i\in I}B_i$ has cardinality at least that of $I$.

Let $i_1 = 1$ and choose an element $b_1\in B_{i_1}$.
More generally, suppose $i_1,\dots, i_k$ and $b_1,\dots, b_k$ are defined.
Then for $I_k = \{1,\dots, N\}\backslash\{i_1,\dots, i_k\}$, the set $\bigcup_{i\in I_k} B_i$ has at least $N-k$ elements.
Choose $b_{k+1}\in \bigcup_{i\in I_k}B_i$ and let $i_{k+1}\in I_k$ with $b_{k+1}\in B_{i_{k+1}}$.
Then the function $\sigma$ defined by
$$\sigma: i_k\mapsto b_k,\quad 1\leq k\leq N$$
is a permutation and by definition $\sigma(k)\neq \sigma_j(k)$ for all $1\leq j \leq m$ and $1\leq k\leq N$.
Hence $C\cup\{\sigma\}$ is a clique, contradicting the maximality of $C$.
\end{proof}

\subsection{Latin Squares and Disconnected Maximal Cliques}
In this section, we provide a novel correspondence between pairs of orthogonal maximal cliques in $X_N$ and pairs of orthogonal Latin squares.
We start by reviewing the basic definition.

An $N\times N$ \vocab{Latin square} is an $N\times N$ matrix $A$ with each number $1,\dots, N$ appearing exactly once in every row or column.
Two Latin squares $A$ and $B$ are said to be \vocab{orthogonal} if the set of ordered pairs $\{(A_{ij},B_{ij}): 1\leq i,j\leq N\}$ has exactly $N^2$ elements.
In other words, when we overlay $A$ and $B$, every ordered pair $(a,b)$ with $1\leq a,b\leq N$ shows up exactly one time.
More generally, a collection of Latin squares is said to be \vocab{mutually orthogonal} if every pair in the collection is an orthogonal pair.

Given a pair $(A,B)$ of orthogonal $N\times N$ Latin squares, we define a pair $(C,\wt C)$ of disconnected ordered maximal cliques in $X_N$ by setting
$$C = \{\sigma_1,\dots,\sigma_N\},\quad\text{and}\quad \wt C = \{\omega_1,\dots,\omega_N\},$$
where here $\sigma_j$ and $\omega_j$ are the permutations defined by
$$\sigma_j: A_{jk}\mapsto B_{jk},\quad\text{and}\quad \omega_j: A_{kj}\mapsto B_{kj},\quad \text{for all}\ \ 1\leq j,k\leq N.$$
Call this correspondence $\Gamma : (A, B) \mapsto (C, \wt{C})$.

\begin{remk}
There is a different correspondence between Latin squares and ordered maximal cliques in $X_N$ that might appear more obvious, where we associate each row with the permutation it defines.  However, this correspondence does not send orthogonal pairs of Latin squares to disconnected cliques.  
\end{remk}

We will show that $\Gamma$ is a bijection.
To start, notice that differences between vertices in disconnected maximal cliques must each fix a unique point.
\begin{lem} \label{lem:nearderangements}
Suppose that $C = \{\sigma_i\}_{i=1}^N$ and $\wt C = \{\omega_i\}_{i=1}^N$ are two disconnected maximal cliques.  Then for any $1\leq i,j\leq N$, the permutation $\sigma_i\omega_j^{-1}$ has exactly one fixed point.
\end{lem}
\begin{proof}
For each $1\leq i,j\leq N$, let $B_{ij} = \{1\leq a\leq N: \sigma_i\omega_j^{-1}(a) = a\}$.
Since $C$ and $\wt C$ are disconnected, $B_{ij}$ is nonempty.
It follows that $\bigcup_j B_{ij}$ has at least $N$ elements and at most $N$ elements.  Moreover, if $j\neq k$, then $\omega_j\omega_k^{-1}$ is a derangement and consequently $B_{ij}\cap B_{ik}=\varnothing$.
It follows that each $B_{ij}$ is a singleton set.
\end{proof}

For convenience, we will call a permutation that fixes exactly one element a \textbf{near-derangement}. 

Using the previous lemma, we can establish a correspondence from ordered maximal cliques to orthogonal Latin squares inverting the previous correspondence.
\begin{thm}
The correspondence $\Gamma: (A,B)\mapsto (C,\wt C)$ is a bijection between pairs $(A,B)$ of $N\times N$ orthogonal Latin squares and pairs of disconnected ordered maximal cliques.
\end{thm}
\begin{proof}
Suppose that $C=\{\sigma_i\}_{i=1}^N$ and $\wt C=\{\omega_i\}_{i=1}^N$ are two disconnected ordered maximal cliques.
For any pair of integers $1\leq i,j\leq N$, let $a_{ij}$ and $b_{ij}$ be the unique integers in $\{1,\dots,N\}$ satisfying
$$\sigma_i(a_{ij}) = b_{ij}\quad\text{and}\quad\omega_j(a_{ij}) = b_{ij}.$$
Define $A$ and $B$ to be the matrices with entries $A_{ij} = a_{ij}$ and $B_{ij} = b_{ij}$.
Since $\sigma_i\sigma_k^{-1}$ is a derangement, $a_{ij}\neq a_{kj}$ for $i\neq k$.  Similarly, $a_{ij}\neq a_{ik}$, $b_{ij}\neq b_{kj}$ and $b_{ij}\neq b_{ik}$.
Therefore $A$ and $B$ are Latin squares.
If $a_{ij}=a_{k\ell}$ and $b_{ij}=b_{k\ell}$, then $\sigma_i\sigma_k^{-1}$ fixes $b_{k\ell}$, which is a contradiction.
Therefore $A$ and $B$ are orthogonal.
It is easy to check that the function $\Omega: (C,\wt C)\rightarrow (A,B)$ and $\Gamma$ are inverses.  Thus $\Gamma$ is a bijection.
\end{proof}

\subsection{Eigendata and Spectral Clustering}
One really appealing idea to come out of the relation between orthogonal Latin squares and disconnected maximal cliques is our ability to rephrase the problem of finding mutually orthogonal Latin squares as a graph clustering problem.
Graph clustering is a standard problem in network analysis where we seek to partition the vertices of a graph into cliques or more general collections of vertices called \vocab{clusters} with the property that inside each cluster, the induced subgraph has more edges relative to the number of edges occurring between different clusters.

Recall that the \vocab{Laplacian} of a graph $X$ is $L=D-A$, where $D$ and $A$ are the degree matrix and adjacency matrix of $X$, respectively.
One very effective technique for graph clustering is \vocab{spectral clustering}, wherein the eigenvectors and eigenvalues of the Laplacian matrix are leveraged in order to find natural splits in the graph structure.
The intuition for this comes from the characterization of the kernel of the Laplacian.

The indicator functions of the connected components of a disconnected graph form a basis of the kernel of its Laplacian.
Therefore if $C$ is a part of a graph $X$ which is strongly interconnected, but weakly connected to the rest of the graph, the indicator function of $C$ should look almost like an element of the kernel of the Laplacian.
In particular, we expect its expansion in terms of the eigendata of the Laplacian to be concentrated in the eigenvalues close to zero.
Alternatively, it should be close to orthogonal to the eigenvectors with largest eigenvalue.

The derangement graph is a Cayley graph with the property that the generators of the edge set (derangements) are closed under conjugation.
In this case, the matrix entries of irreducible representations form eigenfunctions of the Laplacian $L_N$ of $X_N$, viewed as an operator acting on $L^2(S_N)$.
There are several sources of this fact about the spectral data of Cayley graphs for representations over $\mathbb{C}$ \cite{babai,diaconis}.
The result may be extended to other fields with a little bit of care and some additional assumptions about the characteristic of the field.  For completeness, we include a proof of this fact over an arbitrary field now.

\begin{thm}
Let $\F$ be a field and $\pi: S_N\rightarrow M_r(\F)$ be a representation of $S_N$ over $\F$, with $\pi$ irreducible over the algebraic closure $\overline \F$.
Then the matrix entries $\pi_{jk}: S_N\rightarrow \F$ are all eigenfunctions of the Laplacian $L_N$ of the derangement graph $X_N$ of $S_N$.
If $\text{char}(\F)\nmid r$, the corresponding eigenvalue is independent of $j,k$ and given by $\lambda_\pi = \lvert \Delta_N\rvert - \frac{1}{r}\sum_{\sigma\in \Delta_N} \tr(\pi(\sigma))$, where here $\Delta_N$ is the set of derangements in $S_N$.
\end{thm}
\begin{proof}
Suppose that $\pi: S_N\rightarrow M_r(\mathbb{\F})$ is an irreducible representation over $\overline{\F}$.
The Laplacian of the derangement graph is $L_N = \lvert\Delta_N\rvert I-A_N$ for $A_N$ the adjacency matrix of $X_N$.
The action of $L_N$ on the function space $L^2(S_N)$ is given by
$$(L_N f)(\tau) = \sum_{\sigma\in \Delta_N} f(\sigma\tau).$$

Consider the matrix $Q = \sum_{\sigma\in\Delta_N}\pi(\sigma)$.
The $(j,k)$-entry of $\pi$ satisfies
$$(L_N\pi_{jk})(\tau) = \sum_{\sigma\in\Delta_N}\pi_{jk}(\sigma\tau) = \left(Q\pi(\tau)\right)_{jk}.$$
Since $\Delta_N$ is closed under conjugation,
$$Q = \sum_{\sigma\in \omega^{-1}\Delta_N\omega}\pi(\sigma) = \sum_{\sigma\in\Delta_N}\pi(\omega\sigma\omega^{-1}) = \pi(\omega)Q\pi(\omega)^{-1}$$
for all $\omega\in S_N$.
Consequently, the eigenspaces of $Q$ are $\pi$-invariant submodules of $\overline \F^r$.
Since $\pi$ is irreducible, there are no nontrivial proper submodules.
Thus any eigenspace must be the whole space, making $Q = cI$ for some $c\in \F$.
Moreover, by taking the trace of $Q$, we find $c = \frac{1}{r}\sum_{\sigma\in \Delta}\tr(\pi(\sigma))$.
It follows that $L_N\pi_{jk} = \lambda_\pi \pi_{jk}$ with $\lambda_\pi$ defined as in the statement of the theorem.
\end{proof}

In the case of the derangement graph $X_N$, the largest real eigenvalue of the Laplacian is $\lambda_{std} = D_N\left(1+\frac{1}{N-1}\right)$, where $D_N$ is the number of derangements of $S_N$.
The associated eigenspace is spanned by the matrix entries of the standard representation of $S_N$.
This suggests that we should consider disconnected cliques through the lens of their projections onto the eigenspace of the standard representation.

Explicitly, we are motivated to consider the projection, $P_{std}$, of functions onto the eigenspace of the standard representation (non-normalized, as the normalization would involve dividing by $N$, and we will soon be considering finite fields):
$$P_{std}: L^2(S_N)\rightarrow L^2(S_N),\quad P_{std}(f): \tau\mapsto \sum_{\sigma \in S_N} \sum_{1\leq j,k\leq N}f(\sigma)(\pi_{std}(\sigma))_{jk}(\pi_{std}(\tau))_{jk}.$$

In practice, no information is lost by projecting onto the span of the eigenspace of the standard representation and the trivial representation, and this so-called \textbf{natural representation} will simplify the exposition considerably. We will thus work with the (non-reduced) natural representation, which associates each permutation with the corresponding permutation matrix in $M_N(\F)$.
In this case, the indicator function $1_B$ of a subset $B\subseteq S_N$ satisfies
$$P_{nat}(1_B): \tau\mapsto \sum_{\sigma\in B} \tr(\pi_{nat}(\sigma\tau^{-1})) = \sum_{\sigma\in B} n(\sigma;\tau),$$
where here
\begin{equation}\label{eqn:n}
n(\sigma;\tau) = \text{\# (fixed points of $\sigma\tau^{-1}$)}.
\end{equation}

\begin{remk} Here, spectral clustering led us to focus on the standard (or natural) representation. 
However, there is great potential utility in projections onto other eigenspaces for further explorations into mutually orthogonal Latin squares of higher order. 
\end{remk}

\subsection{Modular Obstructions to Existence}

A key insight appearing in \cite{stinson} is that working over $\mathbb{F}_2$ leads to interesting obstructions to the existence of pairs of orthogonal Latin squares.
This leads us to consider the \emph{modular} natural representation of $S_N$ over a finite field $\F$. 
When the characteristic of $\F$ does not divide $N$, the dimension of the image $P_{nat}(L^2(S_N))$ of the projection $P_{nat}$ is $(N-1)^2+1$.
However, in the ``modular setting" when the characteristic divides $N$, this dimension suddenly drops to $(N-1)^2+4-2N$.
This leads to interesting linear dependencies between the projections of elements of $L^2(S_N)$ that do not exist in the non-modular setting.
Note that as an abuse of notation, for any field $\F$ we will still write $L^2(S)$ to represent $\F$-valued functions defined on a finite set $S$, even when $\F$ has positive characteristic.

To start, suppose that we have a collection of $r$ mutually disconnected maximal cliques $C_1,\dots, C_r$.
Then we consider the span of the collection
$$\mathcal B = \left\lbrace P_{nat}(1_\sigma): \sigma\in\bigcup_i C_i\right\rbrace.$$
The sum over any maximal clique $C_i$ of the projections is the constant function:
$$\sum_{\sigma\in C_i}P_{nat}(1_\sigma):\tau\mapsto N.$$
If the characteristic of $\F$ divides $N$, this gives us $r$ linear dependencies coming from the $r$ cliques, which we call the \vocab{clique dependencies}.  The clique dependencies show the subspace spanned by $\mathcal B$ will have dimension at most $r(N-1)$ over $\F$.

However, in this modular setting, one can sometimes show that the dimension must be even smaller.  This leads to \vocab{modular obstructions}, additional conditions on expressions calculating fixed points modulo this characteristic.
This is the main utility of the following theorem.
\begin{thm}\label{modular obstruction}
Let $\F$ be a finite field whose characteristic divides $N$.
Suppose that $\{C_i\}_{i=1}^r$ is a collection of disconnected maximal cliques in $X_N$.  Then the span of
$$\mathcal B = \left\lbrace P_{nat}(1_\sigma): \sigma\in\bigcup_{i=1}^{r} C_i\right\rbrace\subseteq L^2(S_N)$$
has dimension at most $(N^2-4N+r+5)/2$.
\end{thm}
Before proving this theorem, it is worth noting the kinds of obstructions this provides when we are dealing with just a pair of disconnected maximal cliques $C$ and $\wt C$ (i.e., when $r=2$).
In this case, for the obstruction to determine any kind of additional relation, we must have $2\leq N\leq 6$.
Then the number of elements of $\mathcal B$ will be larger than the dimension of its span by more than one, leading to linear dependencies other than the clique dependencies.
The number of expected non-clique dependencies is summarized in Table \ref{tab:num deps}.
Once $N$ exceeds $6$, the guaranteed number of dependencies falls to zero.

\begin{table}[h]
    \centering
    \begin{tabular}{c|c|c|c}
      $N$ & $|\mathcal B|$ & $\dim\vspan\mathcal B$ & \# non-clique dep.\\\hline
       3 & 6  & $\leq 2$  & $\geq 2$\\
       4 & 8  & $\leq 3$  & $\geq 3$\\
       5 & 10 & $\leq 6$  & $\geq 2$\\
       6 & 12 & $\leq 9$  & $\geq 1$
    \end{tabular}
    \caption{Expected number of non-clique dependencies for a pair of disconnected maximal cliques.  Each dependency induces a modular obstruction to the existence of the cliques.}
    \label{tab:num deps}
\end{table}

\begin{remk}
When $r > 2$, we obtain interesting dependencies for higher values of $N$.
For example, the corresponding modular obstructions may be a route to determining the existence of larger collections of mutually orthogonal Latin squares for $N=10$, which is an open problem.
\end{remk}

Each dependency corresponds to a function $f: \bigcup_{i} C_i\rightarrow \F$, which we may assume is not constant on any clique (since we excluded clique dependencies), satisfying
$$\sum_{i=1}^r\sum_{\sigma\in C_i}f(\sigma) P_{nat}(1_\sigma) = 0\in \F.$$
Evaluating this on $\tau\in S_N$ gives
$$\sum_{i=1}^r\sum_{\sigma\in C_i}f(\sigma) n(\sigma;\tau) = 0\mod\text{char}(\F),$$
for $n(\sigma;\tau)$ defined in Equation \eqref{eqn:n}.
We will explore the role of each of these obstructions in more detail in the next sections.

In order to prove Theorem \ref{modular obstruction}, we need to first show that the dimension of the image $P_{nat}(L^2(S_N))$ of the projection operator $P_{nat}$ drops in the modular setting.
\begin{thm}
Suppose that $\text{char}(\F)$ divides $N$.  Then
$$\dim P_{nat}(L^2(S_N)) = (N-1)^2-2N + 4.$$
\end{thm}
\begin{proof}
Let $V\subseteq M_N(\F)$ be the subspace of matrices whose row and column sums
are all the same.
This is the $\F$-span of the set of permutation matrices $\{\pi_{nat}(\sigma): \sigma\in S_N\}$.
The $\F$-vector space $M_N(\F)$ has a natural bilinear form $\langle A,B\rangle = \text{tr}(AB^T)$. 
The dimension of $V$ is $(N-1)^2+1$, so by the rank-nullity theorem, the dimension of the orthogonal complement $V^\perp$ in $M_N(\F)$ is $2N-2$.

Now consider the set
$$W = \text{span}\{R_2-R_1,\dots, R_N-R_1,C_2-C_1,\dots, C_N-C_1\},$$
where $R_j$ and $C_j$ are the matrices with all $1$'s in row $j$ or column $j$, respectively, and zeros elsewhere.
The dimension of $W$ is $2N-3$ and $W\oplus\text{span}\{R_1-C_1\}\subseteq V^\perp$ so $W\oplus\text{span}\{R_1-C_1\}=V^\perp$.
Since we are working in a field where $N=0$, we also see that $W\subseteq V$, and since $R_1-C_1\notin V$, we must have $W=V\cap V^\perp$.

If $f\in L^2(S_N)$, then $P_{nat}(f)$ is the function
$$P_{nat}(f): \tau\mapsto \sum_{\sigma\in S_N} f(\sigma)\tr(\pi_{nat}(\sigma\tau^{-1})).$$
The binlinear form on $M_N(\F)$ combined with $\pi_{nat}$ induces a linear transformation
$$V\mapsto L^2(S_N),\ \ A\mapsto f_A,\ \ \text{where}\ \ f_A(\tau) := \tr(A\pi_{nat}(\tau)^T).$$
In particular, the image of this transformation is $P_{nat}(L^2(S_N))$.
Since the permutation matrices span $V$, the kernel is exactly $V\cap V^\perp=W$.
Thus, it induces an $\F$-vector space isomorphism $V/W\cong P_{nat}(L^2(S_N))$.
The statement of the theorem follows immediately.
\end{proof}

Using this, we can now provide a simple proof of Theorem \ref{modular obstruction}.
\begin{proof}[Proof of Theorem \ref{modular obstruction}]
Consider the vector space
$$V = \text{span}\left\lbrace (P_{nat}(1_\sigma),1_{C_1}(\sigma),\dots, 1_{C_r}(\sigma)): \sigma\in \bigcup_jC_j\right\rbrace\subseteq P_{nat}(L^2(S_N))\oplus \F^r.$$
Let $\vec 1\in \F^r$ be the vector of all $1$'s.
If $\sigma,\wt\sigma\in \bigcup_jC_j$ belong to the same clique, then the inner product of $P_{nat}(1_\sigma)$ and $P_{nat}(1_{\wt{\sigma}})$ is $0$ in $\F$.
Otherwise, if they belong to different cliques, the inner product is $1$.
Thus if $(f,\vec v)\in V$, then $(f,\vec v-\vec 1)\in V^\perp$, so $\dim V^\perp \geq \dim V$.
It follows by rank-nullity that
$$2\dim(V)\leq \dim(V)+\dim(V^\perp) = \dim P_{nat}(L^2(S_N))+r.$$
Thus
$$\dim(V) \leq (N^2+r-4N+5)/2.$$
\end{proof}

\section{Structure of Maximal Cliques for Small $N$}
\label{section:SmallN}
While the maximal clique structures in the derangement graphs when $N \leq 5$ are digestible with software (or in some cases, by hand calculation), and the composition of orthogonal Latin squares is well-understood, these cases provide a fertile testing ground for the techniques presented here, and perhaps motivation for further analysis of less charted territory.

\subsection{Maximal cliques in $X_3$}
The existence of two non-clique dependencies in $L^2(S_3)$ from Theorem \ref{modular obstruction} leads immediately to the following corollary.

\begin{cor} \label{cor:X3}
If there exist disconnected maximal cliques $C$ and $\wt C$ in $X_3$, then there exist two distinct non-constant functions $f_1,f_2: C\cup\wt C\rightarrow \{0,1, 2\}$ each satisfying the property that for all $ \tau\in S_3$,
$$\sum_{\sigma\in C\cup \wt C} f_i(\sigma)n(\sigma;\tau) = 0\mod 3.$$
\end{cor}


In this case, of course, we know that there \emph{do} exist two disconnected maximal cliques, $C$ and $\wt{C}$, and that the six permutations comprising them account for the entirety of $S_3$. 
We may, however, see conditions on the functions $f_1$ and $f_2$ that will prove useful later.

Observe that we may assume $e \in C$, so that all permutations in $C$ have zero fixed points (mod 3), and those in $\wt{C}$ have one. Then, in particular, $\sum_{\sigma \in \wt{C}} f_i(\sigma) = 0 \mod 3$. 

By taking $\tau$ to be any near-derangement from $\wt{C}$, we could conclude that $C\tau^{-1}$ consists of near-derangements, $\wt{C}\tau^{-1}$ consists of the identity and two derangements, so that $\sum_{\sigma \in C} f_i(\sigma) = 0 \mod 3$, as well.

This is, naturally, the full extent of the information that can be gained; exactly two such functions do exist (up to constant multiples, and cyclic permutation), and this provides no obstruction to the existence of the two cliques.


\subsection{Maximal cliques in $X_4$} 
The analogue of Corollary \ref{cor:X3} for $X_4$ is the following.

\begin{cor}
\label{cor:X4}
If there exist disconnected maximal cliques $C$ and $\wt C$ in $X_4$, then there exist three distinct proper subsets $R_1,R_2,R_3\subset C\cup\wt C$, each satisfying the property that for all $\tau\in S_4$,
$$\sum_{\sigma\in R_i} n(\sigma;\tau) = 0\mod 2.$$
\end{cor}

Suppose that two disconnected maximal cliques $C$ and $\wt C$ exist. In this case, we can glean a great deal of information about these subsets $R_i$, and, in turn, this imparts a complete picture of the maximal cliques in $X_4$.

Let us isolate one of the $R_i$, and let $R = R_i \cap C$ and $\wt R = R_i \cap \wt C$. By translating, we can assume $e\in R$. Then, as in the case of $X_3$, $R$ consists of the identity and derangements, and $\wt R$ consists of near-derangements. By taking $\tau$ to be the identity and one of the near-derangements of $\wt R$, respectively, we see that both \begin{equation}
    \label{eqn:wtR}
    \sum\limits_{\sigma \in \wt R}n(\sigma;e) = 0 \mod 2, \mbox{ and}
\end{equation} 
\begin{equation}
    \label{eqn:R}
    \sum\limits_{\sigma \in R}n(\sigma;e) = 0\mod 2.
\end{equation}

In particular, then, $R$ and $\wt R$ each have two elements, and satisfy:
\begin{itemize}
\item $R = \{e,\delta\}$ for some derangement $\delta$, 
\item $\wt R = \{\eta_1, \eta_2\}$ for some near-derangements $\eta_1, \eta_2$,
\item $\eta_1\delta^{-1}$ and $\eta_2\delta^{-1}$ are near-derangements and $\eta_1\eta_2^{-1}$ is a derangement.
\end{itemize}

For $j= 1, 2$, let $a_j$ and $c_j$ be the fixed points of $\eta_j$ and $\eta_j\delta^{-1}$, respectively.
This implies that $\delta(b_j) = c_j$ and $\eta_j(b_j) = c_j$ for some $b_j$.  Since $\delta$ is a derangement, it must be a 4-cycle or a 2,2-cycle. We consider the case where $\delta$ is a 4-cycle. Up to inversion, $\delta$ must take the form of $(a_1 \enspace x \enspace a_2 \enspace y)$ or $(a_1 \enspace a_2 \enspace x \enspace y)$ for some values $x,y\in \{1, 2, 3, 4\} \setminus \{a_1,a_2\}$. Now, consider $\tau_1 = (a_1 \enspace a_2)$ and $\tau_2 = (a_1 \enspace x)$. Observe that for $k=1, 2$, 
\[\sum\limits_{\sigma \in R \cup \wt{R}} n(\sigma;\tau_k) = n(\delta; \tau_k) \mod 2.\]
These $\tau$ restrict the form of $\delta$ in a contradictory manner, as neither form results in an even number of fixed points for both $\delta\tau_1^{-1}$ and $\delta\tau_2^{-1}$.
We conclude that $\delta$ cannot be a 4-cycle.

Now, we note that same argument applies for each of $R_1, R_2$, and $R_3$. All three of the subsets cannot contain $e$ and $\delta$, or the resulting dependencies together with the clique dependency coming from $\wt C$ would imply that, for each $\tau \in S_4$, $n(e;\tau) + n(\delta;\tau) = 0 \mod 2$, which cannot be true. We conclude that $C$ must contain at least two 2,2-cycles, from which it follows that $C$ contains all three 2,2-cycles. That is, there is a unique maximal clique in $X_4$ containing the identity with a disconnected partner maximal clique.

Finally, the same argument shows that $\wt{C}\eta_1^{-1} = C$, so that any maximal clique disconnected from $C$ is simply a translation of $C$. Thus, we can verify the existence of precisely three mutually orthogonal Latin squares of order 4, all related by translations.

\subsection{Maximal cliques in $X_5$}

In the case $N=5$, an understanding of the structure of disconnected pairs of maximal cliques can be obtained directly by elementary means, as shown below.
Even so, it is interesting to view it from the point of view of the modular obstructions in $X_5$, which we examine a posteriori.

The analog of Corollary \ref{cor:X3} for $X_5$ is the following.

\begin{cor}\label{cor:X5}
If there exist disconnected maximal cliques $C$ and $\wt C$ in $X_5$, then there exist two distinct non-constant functions $f_1,f_2: C\cup\wt C\rightarrow \{0,\dots, 4\}$ each satisfying the property that for all $\tau\in S_5$,
$$\sum_{\sigma\in C\cup \wt C} f_i(\sigma)n(\sigma;\tau) = 0\mod 5.$$
\end{cor}

The first thing we can show is that a maximal clique with the identity for which there exists a disconnected maximal clique, cannot contain an odd derangement.
We start with a weaker statement.
\begin{lem}
Suppose $C$ and $\wt C$ are two disconnected maximal cliques in $X_5$ and that $C$ contains the identity.  
Then $C$ does not contain four odd derangements.
\end{lem}
\begin{proof}
Suppose $C = \{\sigma_i\}_{i=1}^5$ with $\sigma_5 = e$ and $\sigma_1,\dots,\sigma_4$ all odd derangements.
Then we may choose an enumeration $a_1,\dots,a_5$ of $\{1,\dots,5\}$ with
$$\sigma_1 = (a_5a_1)(a_2a_3a_4),\quad \sigma_2 = (a_5a_2)(a_3a_1a_4),$$
$$\sigma_3 = (a_5a_3)(a_2a_4a_1),\quad \sigma_4 = (a_5a_4)(a_2a_1a_3).$$
If $\omega$ is an element of $\wt C$, then $\sigma_i\omega^{-1}$ must be a near-derangement for all $i$.
The only elements of $S_5$ with this property are 
$$(a_1a_2)(a_3a_4),\quad (a_1a_3)(a_2a_4),\quad\text{and}\quad (a_1a_4)(a_2a_3).$$
Since $\wt C$ must contain five elements, this is a contradiction.
\end{proof}

Using this lemma, we can immediately prove that $C$ cannot contain any odd derangements at all.

\begin{lem}
Suppose $C$ and $\wt C$ are two disconnected maximal cliques in $X_5$ and that $C$ contains the identity.  
Then $C$ does not contain any odd derangements.
\end{lem}
\begin{proof}
Suppose that $C$ contains an odd derangement $\sigma$.  If $C$ contains fewer than $3$ odd derangements, then we can replace $C$ with $C\sigma^{-1}$ and $\wt C$ with $\wt C\sigma^{-1}$, obtaining two new disconnected maximal cliques with $C \sigma^{-1}$ having at least three odd derangements.
Thus without loss of generality, we may assume $C = \{\sigma_i\}_{i=1}^5$ with $\sigma_5 = e$ and $\sigma_1,\sigma_2,\sigma_3$ odd derangements.
However, any maximal clique in $X_5$ containing the identity and at least three odd derangements, must contain four odd derangements. By the previous Lemma, this is a contradiction.
\end{proof}

As a consequence, we see that if $C$ and $\wt C$ are maximal cliques in $X_5$ and $C$ contains the identity, then $C$ must be made up of only the identity and four $5$-cycles.
Likewise, $\wt C$ must be a translation of a clique made up of only the identity and four $5$-cycles.
We show that these five cycles must be all powers of one-another.
\begin{lem}
Suppose that $C$ is a maximal clique in $X_5$ which contains the identity.
Then
$$C = \{e,\sigma,\sigma^2,\sigma^3,\sigma^4\},$$
for some five-cycle $\sigma$.
\end{lem}
\begin{proof}
By the previous lemma, we know that $C = \{\sigma_i\}_{i=1}^5$ for some permutations with $\sigma_5 =e$ and $\sigma_1,\dots,\sigma_4$ all $5$-cycles.
As $C\tau^{-1}$ is always a maximal clique,
$$\sum_{i=1}^5 n(\sigma_i;\tau) = 5$$ 
for all $\tau\in S_5$.
Taking $\tau=\sigma_1^{-1}$, this implies that there exists $j$ with
$$n(\sigma_j;\sigma_1^{-1}) \geq 2.$$
This means we have
$$\sigma_1 = (a_1a_2a_3a_4a_5),\ \ \text{and}\ \ \sigma_j = (a_3a_2a_1xy),$$
for some enumeration $a_1,\dots,a_5$ of $\{1,\dots,5\}$ and $x,y\in\{x_4,x_5\}$ distinct.
Since $\sigma_1\sigma_j^{-1}$ is a derangement, we must have $x=a_5$ and $y=a_4$. Therefore $\sigma_j = \sigma_1^{-1}$.
The only other $5$ cycles that are derangements of $\sigma_1$ and $\sigma_1^{-1}$ are
$$\sigma_1^2 = (a_1a_3a_5a_2a_4),\ \ \text{and}\ \ \sigma_1^3 = (a_1a_4a_2a_5a_3).$$
This completes the proof.
\end{proof}

To summarize, $C$ must consist of powers of a $5$-cycle $C = \{\sigma^k\}_{k=0}^4$.
Likewise, $\wt C$ is a translation of the powers of a $5$-cycle.
In particular, we can take $C = \{\sigma_i\}_{i=1}^5$ and $\wt C = \{\omega_i\}_{i=1}^5$ be two disconnected maximal cliques with $\sigma_i = \sigma^i$, and $\omega_i = \omega^i\alpha$ for some $5$-cycles $\sigma$ and $\omega$ and some permutation $\alpha$. 

The essence of Corollary \ref{cor:X5} is that the four-dimensional $\mathbb F_5$-subspace of $L^2(C\cup\wt C)$,
$$V = \vspan\{f_1,f_2,1_C,1_{\wt C}\},$$
has the property that $n(\cdot;\tau)\in V^\perp$ for all $\tau\in S_5$.
This implies that the space 
$$U = \vspan\{n(\cdot;\tau): \tau\in S_5\}$$
is at most six-dimensional.
The structure of the cliques above implies that the space
$$U|_C := \{f|_C: f\in U\}$$
defined by restricting functions on $C\cup\wt C$ to $C$ is three dimensional.
In particular, if $\sigma = (a_1a_2a_3a_4a_5)$, then $U_C$ is spanned by the restrictions of $n(\cdot;\omega_1) = 1_C$, $n(\cdot;(a_1a_2))$, and $n(\cdot; (a_1a_2a_3))$.
Similarly, $\dim U|_{\wt C} = 3$, and it follows that $\dim U\leq 6$.
In this way, the structure of the cliques that we discovered makes the modular obstruction obvious for $N=5$. 

\section{Euler's 36 Officer Problem}
\label{section:Euler36}
\subsection{Maximal cliques in $X_6$}
Throughout this section, we will assume that two disconnected maximal cliques $C$ and $\wt C$ exist, resulting ultimately in a contradiction.
We start by exploring the modular obstruction(s) in this situation.

\begin{cor}
If there exist disconnected maximal cliques $C$ and $\wt C$ in $X_6$, then there exists a proper subset $R\subseteq C\cup\wt C$, satisfying the property that for all $\tau\in S_6$,
$$\sum_{\sigma\in R} n(\sigma;\tau) = 0\mod 2.$$
Also there exists a non-constant function $f: C\cup\wt C\rightarrow \{0, 1, 2\}$ satisfying the property that for all $\tau\in S_6$,
$$\sum_{\sigma\in C\cup \wt C} f(\sigma)n(\sigma;\tau) = 0\mod 3.$$
\end{cor}

Isolating the first condition, note that the set $R$ must satisfy the property that $R\cap C$ and $R\cap\wt{C}$ are both nonempty and even, so the number of elements in $R$ is $4,6,8,10$ or $12$.
Also if $R$ satisfies the corollary, so do the sets
$$(C\cap R')\cup (\wt C\cap R),\quad (C\cap R)\cup (\wt C\cap R'),\quad\text{and}\ (C\cup \wt C)\cap R'.$$
Thus without loss of generality, $R$ can be taken to have exactly two elements in $C$ and two elements in $\wt C$.
By translating, we can also assume $e\in R\cap C$.
To summarize, for $N=6$ there must exist a set $R$ with the following list of properties

\begin{itemize}
\item $R = \{e,\delta,\eta_1,\eta_2\}$ for some derangement $\delta$ and near-derangements $\eta_1$ and $\eta_2$, 
\item $\eta_1\delta^{-1}$ and $\eta_2\delta^{-1}$ are near-derangements and $\eta_1\eta_2^{-1}$ is a derangement, and
\item $\sum_{\sigma\in R} n(\sigma; \tau)$ is even for all $\tau\in S_6$.
\end{itemize}

\subsection{Nonexistence proof}
To prove that no $6\times 6$ pair of orthogonal Latin squares exists, we will prove that no subset $R\subseteq S_6$ with the properties outlined in the previous section can exist.

For each $j$, let $a_j$ and $c_j$ be the fixed points of $\eta_j$ and $\eta_j\delta^{-1}$, respectively.
This in particular implies $\delta(b_j) = c_j$ and $\eta_j(b_j) = c_j$ for some $b_j$. 

\begin{lem} \label{lem:allsix}
Assume $R$ exists. Then the sets $D$, $E$, $F$, $G$, and $H$ defined by
\[D = \{a_1, a_2, b_1, b_2, c_1, c_2\},\]
\[E = \{\delta(a_1),\delta^{-1}(a_1),\eta_2(a_1),\eta_2^{-1}(a_1),a_1,a_2\},\]
\[F = \{\delta(a_2),\delta^{-1}(a_2),\eta_1(a_2),\eta_1^{-1}(a_2),a_1,a_2\},\]
\[G = \{c_1, c_2, b_2,\delta(c_2),\eta_1(b_2), \delta(\eta_1^{-1}(c_2))\},\]
\[H = \{c_1, c_2, b_1, \delta(c_1), \eta_2(b_1), \delta(\eta_2^{-1}(c_1))\},\]

are all sets with six elements.
\end{lem}
\begin{proof}
Assume $E$ has fewer than six elements.
Then we can choose $1\leq q\leq 6$ with $q\notin E.$
Consequently, the transposition $\tau = (a_1 \enspace q)$ makes $\delta\tau$ and $\eta_1\tau$ both derangements.
Moreover $\eta_2\tau$ still has a single fixed point, so $\sum_{\sigma\in R} n(\sigma; \tau) = 5$,
which is a contradiction.  

Next, consider the set
\[R\delta^{-1}   = \{e,\wt\delta := \delta^{-1},\wt\eta_1 = \eta_1\delta^{-1},\wt\eta_2 = \eta_2\delta^{-1}\}.\]
The associated fixed-point data is
$$\wt a_1 = c_1,\ \wt a_2 = c_2,\ \wt b_1 = \delta(a_1),\ \wt c_1 = a_1,\ \wt b_2 = \delta(a_2),\ \wt c_2 = a_2.$$
The same arguments as above implies that $\{\wt a_1,\wt a_2,\wt\delta(\wt a_2),\wt\delta^{-1}(\wt a_2),\wt \eta_1(\wt a_2),\wt\eta_1^{-1}(\wt a_2)\}$ is a set of $6$ distinct elements, i.e.,
all six elements of $G$ are distinct. A similar argument proves that $F$ and $H$ each have six elements.

Finally, using the distinctness of the elements of $E$, $F$, $G$, and $H$, together with the facts that $\delta$ is a derangement shows that $D$ has six elements as well.
\end{proof}












We leverage this lemma now to restrict the possible cycle types of $\delta$.
For example, we have the following corollary.
\begin{cor}
Assume $R$ exists.  Then 
the derangement $\delta$ cannot be a $2,2,2$-cycle.
\end{cor}
\begin{proof}
Since $E$ has six elements, $\delta^2(\delta^{-1}(a_1)) = \delta(a_1)\neq \delta^{-1}(a_1)$.  Therefore $\delta$ doesn't have order $2$.
\end{proof}

As a more complicated observation, we have the following.
\begin{lem} \label{lem:b1b2c1c2}
Assume $R$ exists.  Then 
$$b_1\in \{\delta(a_2),\eta_1(a_2)\},\quad\text{and}\quad c_1\in \{\delta^{-1}(a_2),\eta_1^{-1}(a_2)\},$$
and also
$$b_2\in \{\delta(a_1),\eta_2(a_1)\},\quad\text{and}\quad c_2\in \{\delta^{-1}(a_1),\eta_2^{-1}(a_1)\}.$$
\end{lem}
\begin{proof}
From Lemma \ref{lem:allsix}, $b_2\in E$, and $b_2\neq a_1$ and $b_2\neq a_2$.
Moreover, if $b_2=\delta^{-1}(a_1)$ or $b_2=\eta_2^{-1}(a_1)$, then $c_2 = a_1$.
This is a contradiction, so $b_2\in \{\delta(a_1),\eta_2(a_1)\}$.
The other statements are obtained in a similar fashion.
\end{proof}
The previous lemma can be further refined. 
\begin{lem} \label{lem:sigmaaction}
We must have $\delta(a_1) = b_2$ or $\delta(c_2) = a_1$; and $\delta(a_2) = b_1$ or $\delta(c_1) = a_2$.
\end{lem}
\begin{proof}
Suppose $\delta(a_1) \neq b_2$ and $\delta(c_2) \neq a_1$.
Then, by Lemma \ref{lem:b1b2c1c2},  $\eta_2(a_1) = b_2$ and $\eta_2(c_2) = a_1$.  It follows that $\eta_2$ has the three cycle $(c_2 \enspace a_1 \enspace b_2)$, forcing $\eta_2$ to be a $2,3$-cycle.  Now, by Lemma \ref{lem:allsix}, $b_1 \neq c_2$ and $c_1 \neq b_2$, so $\{b_1, c_1\} = \{ \delta(a_1), \delta^{-1}(a_1)\}$. But this means, $\eta_2$ maps $b_1$ to $c_1$, which contradicts $\eta_2\eta_1^{-1}$ being a derangement.
The other statement of the lemma is proved similarly.
\end{proof}

\begin{remk} \label{remk:sigmaa1} This lemma allows us to make a final simplifying assumption.  By swapping $\eta_1$ and $\eta_2$, and possibly by inverting all the elements of $R$, we may assume $b_2 = \delta(a_1)$.
\end{remk}

\begin{cor}
Assume $R$ exists.  Then 
the derangement $\delta$ cannot be a $2,4$-cycle.
\end{cor}
\begin{proof}
Suppose $\delta$ is a $2,4$-cycle.  Since $E$ and $F$ have six elements, both $a_1$ and $a_2$ must belong to the $4$-cycle of $\delta$.
Consequently
$$\delta = (\delta^{-1}(a_1) \enspace a_1 \enspace \delta(a_1) \enspace a_2)(\eta_2^{-1}(a_1) \enspace \eta_2(a_1)).$$
However, since $b_2 = \delta(a_1)$ (see Remark \ref{remk:sigmaa1}), this would imply $a_2 = c_2$, which is a contradiction.
\end{proof}

Now we remark that if $R$ has the properties described above, so too do $R\delta^{-1}$, $R\eta_1^{-1}$, and $R\eta_2^{-1}$. We leverage these facts now.
\begin{cor}
Assume $R$ exists.  Then 
the derangement $\delta$ cannot be a $3,3$-cycle.
\end{cor}
\begin{proof}


In Remark \ref{remk:sigmaa1}, we assumed $b_2 = \delta(a_1)$, so that $c_2 = \delta^2(a_1)$, and $\delta(c_2) = \delta^3(a_1)$ are all different in $G$.
Therefore $a_1$ is not part of a $3$-cycle in $\delta$, so $\delta$ cannot be a $3,3$-cycle.
\end{proof}


\begin{cor}
Assume $R$ exists.  Then 
the derangement $\delta$ cannot be a $6$-cycle.
\end{cor}
\begin{proof}
Suppose that $\delta$ is a $6$-cycle.
The set $R\eta_2^{-1}$ has the same properties as $R$, and contains the derangement $\eta_1\eta_2^{-1}$.
The above arguments then imply $\eta_1\eta_2^{-1}$ must be a $6$-cycle.
This implies that $\eta_1$ and $\eta_2$ have to have opposite signs.
Therefore one must be a $2,3$-cycle and the other must be a $5$-cycle.


If $\eta_j$ is type $2,3$ and $\eta_k$ is type $5$, then the fact that $E$ and $F$ both have six elements implies $\eta_j = (\eta_{j}^{-1}(a_k) \enspace a_k \enspace \eta_j(a_k))(\delta(a_k) \enspace \delta^{-1}(a_k))$.
Recall that $b_2 = \delta(a_1)$.  Therefore if $j=2$, our expression for $\eta_j$ says that $\delta^2(a_1) = \delta(b_2) = \eta_2(b_2) = \delta^{-1}(a_1)$.  This is impossible, since $\delta$ is a $6$-cycle.
Therefore $j=1$. If $\delta(a_2) = b_1$, then $\delta^2(a_2) = \delta(b_1) = \eta_1(b_1) = \eta_1(\delta(a_2)) = \delta^{-1}(a_2)$, which is impossible, as $\delta$ is a 6-cycle. Thus, $\delta(a_2)\neq b_1$, and by Lemma \ref{lem:b1b2c1c2}, $\eta_1(a_2) = b_1$, so $\eta_1^{-1}(a_2) = c_1$, making
$$\eta_1 = (c_1 \enspace a_2 \enspace b_1)(\delta(a_2) \enspace \delta^{-1}(a_2)).$$

The $5$-cycle $\eta_2$ is then given by $\eta_2 = (\eta_2^{-1}(a_1) \enspace a_1 \enspace \eta_2(a_1) \enspace x \enspace y)$ for some values $x,y\in\{\delta(a_1),\delta^{-1}(a_1)\}$.
If $\delta(a_1) = x$, then since $\delta(a_1) = b_2$, we would have 
$\delta^{-1}(a_1) = \eta_2(\delta(a_1)) = \eta_2(b_2) = \delta(b_2) = \delta^2(a_1)$,
which is impossible since $\delta$ has order $6$.
Consequently, the only remaining possibility is that $\eta_2 = (\eta_2^{-1}(a_1) \enspace a_1 \enspace \eta_2(a_1)\enspace \delta^{-1}(a_1) \enspace \delta(a_1))$.

Now, as $\delta(a_1) = b_2$, and $\eta_2(b_2) = \delta(b_2) = c_2$, we have $\eta_2^{-1}(a_1) = \delta^2(a_1)$. Thus, either $\delta = (\delta^{-1}(a_1) \enspace a_1 \enspace \delta(a_1) \enspace \delta^{-1}(a_2) \enspace a_2 \enspace \delta(a_2) )$, or $\delta = (\delta^{-1}(a_1) \enspace a_1 \enspace \delta(a_1) \enspace \delta^{-2}(a_2) \enspace \delta^{-1}(a_2) \enspace a_2 )$. The latter is impossible, as then, from $E$, $\delta^{-1}(a_2) = \eta_2(a_1)$, and so $\eta_2^2(a_1) = \delta^{-1}(a_1) = \delta(a_2) = \eta_1(\delta^{-1}(a_2)) =  \eta_1(\eta_2(a_1))$, which contradicts $\eta_1\eta_2^{-1}$ being a derangement.

Finally then, we consider $b_1$. Observe, $b_1 \neq a_1$, $b_1 \neq b_2 = \delta(a_1)$, $b_1 \neq c_2 = \delta(b_2) = \delta^2(a_1)$, and $b_1 \neq a_2 = \delta^3(a_1)$. Also, $c_1 \neq a_1$, so $b_1 \neq \delta^{-1}(a_1)$, and from $\eta_1$, $b_1 \neq \delta(a_2) = \delta^{-2}(a_1)$. This leaves no remaining options for $b_1$, which completes the contradiction.



\end{proof}

Summarizing the above results, we've proved the following theorem.
\begin{thm}
There do not exist two disconnected maximal cliques in $X_6$.
\end{thm}
\begin{cor}
There do not exist two $6\times 6$ mutually orthogonal Latin squares.
\end{cor}

\section*{Acknowledgements}
\thanks{The research of W.R.C. has been supported by an AMS-Simons Research Enhancement Grant, and RSCA intramural grant 0359121 from CSUF; that of Sam Fleyshman and Sara Anderson was supported 
by CSUF's Math Summer Research Program and a Deland Summer Research Fellowship.}

\bibliographystyle{plain}
\bibliography{DCDG}


\end{document}